\documentclass[12pt]{amsart}
\usepackage{amssymb,amsmath,amsthm}
\usepackage{xcolor}
\usepackage[colorlinks=true,
linkcolor=webgreen,
filecolor=webbrown,
citecolor=webgreen]{hyperref}
\definecolor{webgreen}{RGB}{0,0,1}
\definecolor{recrown}{RGB}{1,.2,.6}
\usepackage{fullpage}
\usepackage{orcidlink}

\begin{document}
\newtheorem{theorem}{Theorem}
\newtheorem{corollary}[theorem]{Corollary}
\newtheorem{lemma}[theorem]{Lemma}
\theoremstyle{definition}
\newtheorem{example}{Example}
\newtheorem*{examples}{Examples}
\newtheorem*{notation}{Notation}
\newtheorem*{remark}{Remark}
\theoremstyle{theorem}
\newtheorem{thmx}{Theorem}
\renewcommand{\thethmx}{\text{\Alph{thmx}}}
\newtheorem{lemmax}[thmx]{Lemma}
\renewcommand{\thelemmax}{\text{\Alph{lemmax}}}
\leftmargin=.5in
\rightmargin=0.5in
\theoremstyle{definition}
\newtheorem*{definition}{Definition}
\title[]{\bf Some factorization results on polynomials having integer coefficients}
\author{Jitender Singh$^{1}$ {\large \orcidlink{0000-0003-3706-8239}}}
\address[1,2]{Department of Mathematics,
Guru Nanak Dev University, Amritsar-143005, India\linebreak
 {\tt sonumaths@gmail.com}}
\author{Rishu Garg$^{2}$}
\date{}
\parindent=0cm
\footnotetext[1]{Corresponding Author email(s): {\tt sonumaths@gmail.com; jitender.math@gndu.ac.in}\\

2010MSC: {Primary 12E05; 30C15; 11C08; 26C10; 26D05; 30A10}\\

}
\maketitle
\begin{abstract}
In this article, we prove some factorization results for several classes of polynomials having integer coefficients, which in particular yield several classes of irreducible polynomials. Such classes of polynomials are devised by imposing some sufficiency conditions on their coefficients along with some conditions on the prime factorization of constant term or the leading coefficient of the underlying polynomials in conjunction with the information about their root location.
\end{abstract}
\section{Introduction.}
The irreducibility criteria due to Sch\"onemnn \cite{S}, Eisenstein \cite{E}, Perron \cite{Perron}, and Dumas \cite{D} are among the most popular classical results concerning irreducibility of polynomials having integer coefficients \cite{Cox}. The irreducibility criteria due to Sch\"onemnn and Eisenstein are easy consequences of the much more general factorization result of Dumas based on Newton-polygon approach. There have been many extensions and generalizations of the aforementioned irreducibility criteria for polynomials having integer coefficients, and the literature is vast; so for a quick survey, the reader may refer to see \cite{SKJS2023}.

In 2013, Weintraub proved the following generalization of the Eisenstein's irreducibility criterion which also rectifies the false claim made by Eisenstein himself in \cite{E}.
\begin{thmx}[Weintraub \cite{W}]\label{th:A}
    Let $f=a_0+ a_{1}z+\cdots+a_n z^m\in \Bbb{Z}[z]$ be a polynomial and suppose there is a prime  $p$ such that $p$ does not divide $a_m$, $p$ divides $a_i$ for $i=0,\ldots,m-1$, and for some $k$ with $0\leq k\leq m-1$, $p^{2}$ does not divide $a_k$. Let $k_0$ be the smallest such value of $k$. If $f(z)=g(z)h(z)$, a factorization in $\mathbb{Z}[z]$, then $\min\{\deg(g),\deg(h)\}\leq k_0$. In particular, for a primitive polynomial $f$, if $k_0=0$, then $f$ is irreducible, and if $k_0=1$ and $f$ does not have a root in $\mathbb{Q}$, then $f$ is irreducible.
\end{thmx}
The irreducibility criterion of Eisenstein is precisely the case $k=0$ of Theorem \ref{th:A}, and such results are based on factorization properties of the coefficients of the given polynomial with respect to a prime divisor. 
On the other hand, the irreducibility criterion due to Perron is based on root location of the underlying polynomial. More precisely, the irreducibility criterion of Perron states that if the coefficients of the monic polynomial $f=z^m+\sum_{i=0}^{m-1}a_iz^i \in \mathbb{Z}[z]$, $m\geq 2$  satisfy
\begin{eqnarray*}
|a_{m-1}| &>& 1+|a_{m-2}|+|a_{m-3}|+\cdots+|a_1|+|a_0|,
\end{eqnarray*}
then $f$ is irreducible in $\mathbb{Z}[z]$. The aforementioned inequality enforces all but one zero of $f$ to lie inside the open unit disk $|z|<1$ in the complex plane, and the remaining zero to lie outside the closed unit disk $|z|\leq 1$. This information about the zeros of $f$ can be utilized in order to deduce the irreducibility of $f$. Interestingly, the irreducibility criterion of Perron itself implies the following extension for non-monic polynomials as follows.
\begin{thmx}\label{th:B}
Let $f = a_0 + a_1z +\cdots+a_m z^m\in \Bbb{Z}[z]$ be a primitive polynomial with $a_0a_m\neq 0$ and $m\geq 2$, such that
\begin{eqnarray*}
|a_{m-1}|&>& 1+ |a_0||a_m|^{m-1}+|a_1||a_m|^{m-2}+\cdots+|a_{m-2}||a_m|.
\end{eqnarray*}
Then $f$ is irreducible in $\mathbb{Z}[z]$.
\end{thmx}
\begin{proof}
Define $g(z)=a_m^{m-1}f(z/a_m)$. It follows that $g\in \mathbb{Z}[z]$. Further, the polynomials $f$ and $g$ have the same number of irreducible factors in $\mathbb{Z}[z]$. In fact we have
\begin{eqnarray*}
g=a_0a_{m}^{m-1}+a_1a_{m}^{m-2}z+\cdots+a_{m-2}a_m z^{m-2}+a_{m-1} z^{m-1}+z^m,
\end{eqnarray*}
which shows that $g$ is a monic polynomial over $\mathbb{Z}$. By the hypothesis, the coefficients of $g$ satisfy 
$|a_{m-1}|>1+|a_0a_m^{m-1}|+ |a_1a_{m}^{m-2}|+\cdots+|a_{m-2}a_m|$. 
Consequently, by the classical irreducibility criterion of Perron, the polynomial $g$ is irreducible in  $\mathbb{Z}[z]$, and thus so does the polynomial $f$.
\end{proof}
Perron \cite{Perron} gave a multivariate version of his irreducibility criterion for polynomials over arbitrary fields, whose detailed proof may be found in \cite{Bonciocat2}. As a generalization of the irreducibility criterion of Perron,  the authors in \cite{JRG2023} proved a factorization result for a class of polynomials having integer coefficients.

In this article, we obtain factorization results for polynomials having integer coefficients satisfying the conditions  analogous to that of the conditions analogous to those of Eisenstein's irreducibility criterion which depend upon prime factorization of the coefficients of the underlying polynomial, or the irreducibility criterion of Perron which are purely based on root location, or the conditions on the prime factorization of the constant term or the leading coefficient in connection with the root location. Our main results are stated  in Sec.~\ref{Sec:2}, where at the end some examples based on the main results are discussed. Finally, the proofs are presented in Sec.~\ref{Sec:3}.  
\section{Main results}\label{Sec:2}
Analogous to Theorem \ref{th:A}, we have another generalization of Eisenstein's irreducibility criterion in a somewhat different setup as follows:
\begin{theorem}\label{th:1}
    Let $f=a_0+ a_{1}z+\cdots+a_m z^m\in \Bbb{Z}[z]$ be a polynomial and suppose there exists a prime  $p$ and  coprime positive integers $k$ and $j$ with $1\leq j\leq m$ such that  $p$ does not divide $a_j$, $p^k$ divides $a_i$ for $i=0,\ldots,j-1$, and $p^{k+1}$ does not divide $a_0$. Then any factorization $f(z)=g(z)h(z)$ of $f$ in $\mathbb{Z}[z]$ has a factor of degree less than or equal to $k_0$, where $k_0\in \{0,m-j\}$. In particular, for a primitive polynomial $f$,  if $j=m$, then $f$ is irreducible, and if $j=m-1$ and $f$ does not have a root in $\mathbb{Q}$, then $f$ is irreducible.
 \end{theorem}
Theorem \ref{th:1} for $j=m$ yields a well known irreducibility criterion which may be found in \cite[Theorem B]{JSSK2021}.

Our next result uses prime factorization of the constant term of $f$ in connection with the root location of the polynomial $f$, which also generalizes one of the irreducibility criteria proved in \cite[Theorem A]{JSSK2021}.
\begin{theorem}\label{th:2}
Let $f = a_0 + a_1z +\cdots+a_m z^m\in \Bbb{Z}[z]$ be a primitive polynomial for which $a_0=\pm p^k d$ for some  positive integers $k$ and $d$ and a prime $p$ not dividing $d$ such that the zeros of $f$ lie outside the closed disk $|z|\leq d$. Suppose there exists an index $j$ with $1\leq j\leq m$ for which $p$ does not divide $a_j$. Then $f$ is a product of at most $\min\{k,j\}$ irreducible factors in $\mathbb{Z}[z]$. In particular, if $k=1$ or $j=1$, then $f$ is irreducible.
\end{theorem}
If instead the prime factorization of the leading coefficient of $f$ is used, then we may have the following factorization result.
\begin{theorem}\label{th:3}
Let $f = a_0 + a_1z +\cdots+a_m z^m\in \Bbb{Z}[z]$ be a primitive polynomial for which $a_m=\pm p^k d$ for some  positive integers $k$ and $d$ and a prime $p$ not dividing $d$, such that the zeros of $f$ lie outside the closed disk $|z|\leq d$.  Suppose there exists an index $j$ with $1\leq j\leq m$ for which $p$ does not divide $a_{m-j}$. If $|a_0/q|\leq |a_m|$, where $q$ is the smallest prime divisor of $a_0$, then $f$ is a product of at most $\min\{k,j\}$ irreducible factors in $\mathbb{Z}[z]$. In particular, if $k=1$ or $j=1$, then $f$ is irreducible.
\end{theorem}
As a generalization of Theorem \ref{th:B} to include the class of polynomials in which a coefficient other than the $(m-1)$th coefficient is dominant, we may have the following result.
\begin{theorem}\label{th:4}
Let $f = a_0 + a_1z +\cdots+a_n z^n\in \Bbb{Z}[z]$ be a primitive polynomial with $a_0a_m\neq 0$ and $m\geq 2$. Let $b$ be a positive divisor of $a_m$ and $\delta$ be a positive real number with $(1/b)\leq \delta\leq 1$ for which there exists an index $j$ with $0\leq j\leq m-1$ such that
\begin{eqnarray*}
|a_{j}|&>&\sum_{0\leq i<j}|a_i||a_{m}|^{j-i}+\sum_{j<i\leq m}|a_i|\delta^{i-j},
\end{eqnarray*}
where the summation over empty set is understood to be zero. Then $f$ is a product of at most $m-j$ irreducible polynomials in $\mathbb{Z}[z]$. In particular, if $j=m-1$, then $f$ is irreducible.
\end{theorem}
 The following corollaries are immediate from Theorem \ref{th:4}.
\begin{corollary}\label{cor:0}
    Let $f = a_0 + a_1z +\cdots+a_n z^n\in \Bbb{Z}[z]$ be a primitive polynomial with $a_0a_m\neq 0$ and $m\geq 2$. Let $b$ be a positive divisor of $a_m$ for which there exists a real number $\gamma\geq |a_m|/b$ and an index $j$ with $0\leq j\leq m-1$ such that
\begin{eqnarray*}
|a_{j}|&>&\sum_{0\leq i<j}|a_i|b^{j-i}\gamma^{j-i}+\sum_{j<i\leq m}\frac{|a_i|}{|a_m|^{i-j}}\gamma^{i-j},
\end{eqnarray*}
where the summation over empty set is understood to be zero. Then $f$ is a product of at most $m-j$ irreducible polynomials in $\mathbb{Z}[z]$. In particular, if $j=m-1$, then $f$ is irreducible.
\end{corollary}
The Corollary \ref{cor:0} in particular yields the main result proved in the paper \cite{JRG2023} for $b=1$.
Taking $b=|a_m|=1/\delta$ in Theorem \ref{th:4} yields the following factorization result.
\begin{corollary}\label{cor:1}
Let $f = a_0 + a_1z +\cdots+a_m z^m\in \Bbb{Z}[z]$ be a primitive polynomial with $a_0a_m\neq 0$ and $m\geq 2$. Suppose there exists an index $j$ with $0\leq j\leq m-1$ for which
\begin{eqnarray*}
|a_{j}|&>&|a_{j+1}/a_m|+\sum_{0\leq i<j}|a_i||a_m|^{j-i}+\sum_{j+1<i\leq m}|a_{i}||a_m|^{-(i-j)},
\end{eqnarray*}
where the summation over empty set is understood to be zero. Then $f$ is a product of at most $m-j$ irreducible polynomials in $\mathbb{Z}[z]$. In particular, if $j=m-1$, then $f$ is irreducible.
\end{corollary}
In \cite{Bonciocat0}, Bonciocat et al. proved an irreducibility criterion for polynomials having integer coefficients in which the leading coefficient is divisible by a large prime power. The following result complements their irreducibility criterion and also improves upon the corresponding factorization result proved by the authors in \cite[Corollary 2]{JRG2023}.
\begin{corollary}\label{cor:2}
Let $f = p^sa_{j-1}z^{j-1}+p^N a_jz^j+\sum_{i=0;~i\neq j-1,j}^m a_i z^i\in \Bbb{Z}[z]$ be a primitive polynomial with $a_0a_{j-1}a_ja_m\neq 0$, $N\geq 1$, $m\geq 2$, $s\geq 0$, $1\leq j\leq m-1$, let $p$ be a prime number, $p\nmid |a_{j-1}a_j|$ such that
\begin{eqnarray*}
p^N|a_j|&>& |a_ma_{j-1}|p^{2s}+\sum_{i=2}^{j}|a_m^{i}a_{j-i}|p^{is}+\sum_{i=j+1}^m \frac{|a_{i}|}{|a_m|^{i-j}}.
\end{eqnarray*}
Then $f$ is a product of at most $m-j$ irreducible polynomials in $\mathbb{Z}[x]$. In particular, if $j=m-1$, then $f$ is irreducible.
\end{corollary}
We now give some examples whose factorization properties may be deduced from Theorems \ref{th:1}-\ref{th:4}, respectively.
\begin{example}
For a prime $p$ and positive integers $m$ and $n$ with $m\geq n\geq 2$, the polynomial
\begin{eqnarray*}
P_1 &=& p^{m-1}\Bigl(\frac{z^{n}-1}{z-1}\Bigr)\pm z^{m}
\end{eqnarray*}
satisfies the hypothesis of Theorem \ref{th:1} for $j=m$, and $k=m-1$. So,  the polynomial $P_1$ is irreducible in $\mathbb{Z}[z]$.
\end{example}
\begin{example}
For a prime $p$ and positive integer $k$, $d$, and $m\geq 2$ with $p\nmid d$, consider the polynomial
\begin{eqnarray*}
P_2 &=& \pm p^k d+a_1z+a_2z^2+\cdots+a_m z^m,
\end{eqnarray*}
where
\begin{eqnarray*}
p^k>m\max\{|a_1|,|a_2|d,\ldots,|a_m|d^{m-1}\}.
\end{eqnarray*}
For $|z|\leq d$ we have
\begin{eqnarray*}
|P_2(z)|\geq p^k d-\sum_{i=1}^m |a_i||z|^i &\geq&d\bigl(p^k-\sum_{i=1}^m \max_{1\leq i\leq m}\{|a_i|d^{i-1}\}\bigr)\\
&=&d\bigl(p^k-m\max_{1\leq i\leq m}\{|a_i|d^{i-1}\}\bigr)>0.
\end{eqnarray*}
Consequently, all zeros of $P_2$ lie outside the closed disk $|z|\leq d$. If $k=1$, then by Theorem \ref{th:2}, the polynomial $P_2$ is irreducible in $\mathbb{Z}[z]$.

On the other hand if  $k\geq 2$, and $p\nmid a_1$, then again $P_2$ is irreducible. Finally, if $k\geq 2$, and $p\nmid a_j$ with $j\geq 2$, then in view of Theorem \ref{th:2}, the polynomial $P_2$ is a product of at most $\min\{2,j\}=2$ irreducible factors in $\mathbb{Z}[z]$.
\end{example}
\begin{example}
For a prime $p$ and positive integer $k$, $d$, and $m\geq 2$ with $p\nmid d$, consider the polynomial
\begin{eqnarray*}
P_3 &=& a_0+a_1z+a_2z^2+\cdots+a_{m-1} z^{m-1}\pm (p^k d)z^{m},
\end{eqnarray*}
where
\begin{eqnarray*}
|a_0|>m\max\{|a_1|d,|a_2|d^2,\ldots,|a_{m-1}|d^{m-1},p^k d^{m+1}\}.
\end{eqnarray*}
Then proceeding as in the preceding example, it can be proved that all zeros of $P_3$ lie outside the closed disk $|z|\leq d$. Now if  $p\nmid a_{m-j}$ for some $j$ with $1\leq j\leq m$ and $|a_0/q|\leq p^kd$, where $q$ is the smallest prime divisor of $a_0$, then by Theorem \ref{th:3}, the polynomial $P_3$ is a product of at most $\min\{k,j\}$ irreducible factors in $\mathbb{Z}[z]$.
\end{example}
\begin{example} For positive integers $m$, $j$, $a$, $b$ with $m\geq 3$, $1\leq j\leq m-1$, and $b<a-b$, the polynomial
\begin{eqnarray*}
P_4 &=& 1\pm az\pm a^2z^2\pm\cdots\pm a^{j-1} z^{j-1}\pm (a^{j}-b^{j}+1)z^{j} \pm bz^m
\end{eqnarray*}
satisfies the hypothesis of Theorem \ref{th:4} with $a_i=a^i$ for $i=0,\ldots,j$, $a_i=0$ for $i=j+1,\ldots,m-1$, $a_m=b$, $\delta=1/b$, since here
\begin{eqnarray*}
|a_{j}|=(a^{j}-b^{j}+1)>b\frac{a^{j}-b^{j}}{a-b}+\frac{1}{b^{m-1-j}}=\sum_{0\leq i<j-1}|a_i|b^{j-i}+\sum_{j<i\leq m}|a_i|\delta^{i-j}.
\end{eqnarray*}
So, the polynomial $P_4$ is a product of at most $m-j$ irreducible polynomials in $\mathbb{Z}[z]$.
\end{example}
\section{Proofs of Theorems}\label{Sec:3}
To prove Theorems \ref{th:1}, we need the following result proved in \cite{JSSK2021}.
\begin{lemmax}[Singh and Kumar \cite{JSSK2021}]\label{L1}
    Let $f=a_0+ a_{1}z+\cdots+a_m z^m$, $g=b_0+b_1z+\cdots+b_nz^n$, and  $h=c_0+c_1z+\cdots+c_{m-n}z^{m-n}$ be nonconstant polynomials in $\Bbb{Z}[z]$ such that $f(z)=g(z)h(z)$. Suppose there exists a prime $p$ and positive integers $k\geq 2$ and $1\leq j\leq m$ such that $p^k$ divides
    $\gcd(a_0,a_1,\ldots,a_{j-1})$, $p^{k+1}$ does not divide $a_0$, and $\gcd(k,j)=1$. If $p$ divides $b_0$ and $p$ divides $c_0$, then $p$ divides $a_j$.
 \end{lemmax}
\begin{proof}[\bf Proof of Theorem \ref{th:1}]
Let $g(z)=b_0+b_1z+\cdots+b_nz^n$  and $h(z)=c_0+c_1z+\cdots+c_{m-n}z^{m-n}$ for some $n$ with $0\leq n \leq m$. Since $f(z)=g(z)h(z)$, we may have
\begin{eqnarray*}
a_i &=& b_0c_i+b_1c_{i-1}+\cdots+b_{i-1}c_1+b_ic_0,~0\leq i\leq m,
\end{eqnarray*}
where we define $b_{n+i}=0=c_{m-n+i}$  for all $i\geq 1$. In particular, for $i=0$,  we have  $a_0=b_0c_0$, and since $p^k$ divides $a_0$, it follows that $p$ divides at least one of $b_0$ and $c_0$.

Suppose that $p$ does not divide $c_0$.  Then $p^k$ divides $b_0$. Since $p^k$ divides $a_i$ for each $i=1,\ldots, j-1$, and since $a_i=b_0c_i+b_1c_{i-1}+\cdots+b_ic_0$, it follows recursively that $p^k$ divides $b_i$ for each $i=0,\ldots,j-1$. Consequently, $p^k$ divides the sum $(b_0c_j+\cdots+b_{j-1}c_1)=a_j-b_jc_0$, and since $p$ does not divide any of $a_j$ or $c_0$, we find that $p$ does not divide $b_j$. Since $f$ is primitive, $p$ can't be a factor of $g$, and since $p$ divides $b_i$ for each $i=0,\ldots,j-1$, we must have $\deg(g)>(j-1)$. So, we must have $\deg(h)\leq (m-j)$.

Now suppose that $p$ divides $b_0$ and $p$ divides $c_0$. Then $k\geq 2$. If each of $g$ and $h$ is a nonconstant polynomial in $\mathbb{Z}[z]$, then by Lemma \ref{L1}, $p$ divides $a_j$, which contradicts the hypothesis of the theorem. So, one of $g$ and $h$ is a constant, which therefore must be a unit,  since $f$ is primitive. Thus, in this case $k_0=0$.
\end{proof}
\begin{proof}[\bf Proof of Theorem \ref{th:2}]
Let $f(z)=f_1(z)\cdots f_r(z)$ be a product of $r$ irreducible polynomials $f_1,\ldots,f_r\in \mathbb{Z}[z]$. Then $p^kd=|a_0|=|f(0)|=|f_1(0)|\cdots |f_r(0)|$, and so, $|f_i(0)|\geq 1$ for each $i$. If $\alpha_i\neq 0$ is the leading coefficient of $f_i$,  then we may write $f_i(z)=\alpha_i\prod_{\theta} (z-\theta)$, where the product runs over all zeros of $f_i$. We then have
\begin{eqnarray*}
|f_i(0)|=|\alpha_i|\prod_{\theta} |\theta|>|\alpha_i| d^{\deg (f_i)}\geq d,
\end{eqnarray*}
for each $i=1,\ldots,r$, which shows that $p$ divides $f_i(0)$ for each $i=1,\ldots,r$, and so, $r\leq k$.
If $k\leq j$, then $\min\{k,j\}=k\geq r$, and we are done.

So, assume that $k>j$. Suppose on the contrary that $r>j$. If we write $f_i=\sum_{t=0}^{\deg (f_i)}a_{it}z^t\in \mathbb{Z}[z]$, then we may have
\begin{eqnarray*}
a_j &=& \sum_{i_1+i_2+\cdots+i_r=j}a_{1i_1}a_{2i_2}\cdots a_{ri_r},
\end{eqnarray*}
where the indices under the summation satisfy $0\leq i_1,\ldots,i_r\leq j$. Since $r>j$, it follows that given any $r$-tuple $i_1,\ldots,i_r$ with $0\leq i_1,\ldots,i_r\leq j$ and $i_1+\cdots+i_r=j$, at least one of the indices $i_1,\cdots,i_r$ is equal to 0. Suppose that $i_t=0$ for some $t\in \{1,\ldots,r\}$. Then $f_t(0)=a_{t0}$, and so, $f_t(0)$ is a factor of $a_{1i_1}a_{2i_2}\cdots a_{ri_r}$. Since $p$ divides $f_i(0)$ for each $i=1,\ldots,r$, it follows that $p$ divides $\sum_{i_1+i_2+\cdots+i_r=j}a_{1i_1}a_{2i_2}\cdots a_{ri_r}=a_j$, which contradicts the hypothesis.
\end{proof}
\begin{proof}[\bf Proof of Theorem \ref{th:3}]
Let $f(z)=f_1(z)\cdots f_r(z)$ be the product of $r$ irreducible polynomials $f_1,\ldots,f_r\in \mathbb{Z}[z]$. Since $1\leq \min\{k,j\}$, we may assume without loss of generality that $r>1$. If $\alpha_i\neq 0$ denote the leading coefficient of $f_i$ for each $i$, then $\pm p^kd=a_m=\alpha_1\cdots \alpha_r$. Further, we may write $f_i(z)=\alpha_i\prod_{\theta}(z-\theta)$, where the product is over all zeros of $f_i$. We then have
\begin{eqnarray*}
|f_i(0)| &=& |\alpha_i|\prod_{\theta}|\theta|>  |\alpha_i|d^{\deg(f_i)},
\end{eqnarray*}
which shows that $|f_i(0)|/(|\alpha_i|d^{\deg(f_i)})>1$  for each $i$. Consequently, $|f_i(0)|>d$ for each $i$. Using these observations for each $i\in\{1,\ldots,r\}$ along with the hypothesis, we may arrive at the following:
\begin{eqnarray*}
|a_0/q| \leq |a_n|&<&|a_n|\Bigl(\frac{|f_1(0)|}{|\alpha_1|d^{\deg(f_1)}}\cdots \frac{|f_r(0)|}{|\alpha_r|d^{\deg(f_r)}}\Bigr)\times \frac{|\alpha_i|d^{\deg(f_i)}}{|f_i(0)|}
= |a_0|\frac{|\alpha_i|}{|f_i(0)|d^{m-\deg(f_i)}},
\end{eqnarray*}
which shows that $|\alpha_i|>|f_i(0)/q|d^{m-\deg(f_i)}\geq d^{m-\deg(f_i)}\geq d$, since $r>1$ enforces $(m-\deg(f_i))\neq 0$; $|f_i(0)/q|\geq 1$, since $|f_i(0)|$ is a nontrivial divisor of $|a_0|$ for each $i$ and $q$ is the smallest prime divisor of $|a_0|$. So, $|\alpha_i|>d$ for each $i$. This in view of the fact that $\pm p^k d=a_m=\alpha_1\cdots \alpha_r$ and that $p$ does not divide $d$ shows that $p$ must divide each $\alpha_i$. This proves that $r\leq k$.

Now on the contrary, assume that $r>j$.  We may write each $f_i=\sum_{t=1}^{\deg (f_i)}a_{it}z^t$ so that $\alpha_i=a_{i\deg(f_i)}$. Since $f(z)=f_1(z)\cdots f_r(z)$, we have $m=\sum_{i=1}^r \deg(f_i)$ and so,
\begin{eqnarray*}
a_{m-j}&=&\sum_{\sum_{t=1}^r i_t=m-j}a_{1i_1}\cdots a_{ri_r}=\sum_{\sum_{t=1}^r(\deg(f_t)-i_t)=j} a_{1i_1}\cdots a_{ri_r},
\end{eqnarray*}
where $0\leq i_t\leq \deg(f_t)$ and $0\leq \deg(f_t)-i_t\leq j$ for each index $i_t$. Since $r>j$, it follows that at least one of the numbers $\deg(f_1)-i_1,\cdots,\deg(f_r)-i_r$ must be equal to zero for every choice of the $r$-tuple $i_1,\ldots,i_r$. Thus for a given   $r$-tuple $i_1,\ldots,i_r$, there exists a $t\in \{1,\ldots,r\}$ for which $\deg(f_t)-i_t=0$, so that
$a_{ti_t}=a_{t\deg(f_t)}=\alpha_t$, the leading coefficient of $f_t$, which is divisible by $p$. Thus $p$ divides each term $a_{1i_1}\cdots a_{ri_r}$, and so, $p$ divides $\sum_{\sum_{t=1}^r i_t=m-j}a_{1i_1}\cdots a_{ri_r}=a_{m-j}$, which contradicts the hypothesis.
\end{proof}
To prove Theorem \ref{th:4}, we use the following result whose ideas were inherent in the paper \cite{JRG2023}.
\begin{lemma}\label{L2}
    If $g=a_0+a_1z+\cdots+a_mz^m\in \mathbb{Z}[z]$ is such that for some index $j\in \{0,\ldots,m-1\}$, $j$ zeros of $f$ lie within the open disk $|z|<1/|a_m|$ and $m-j$ zeros lie outside the closed disk $|z|\leq 1$, then $f$ is a product of at most $m-j$ irreducible factors in $\mathbb{Z}[z]$.
\end{lemma}
\begin{proof}Suppose that $g(z)=g_1(z)g_2(z)\cdots g_{r}(z)$ is the product of $r$ nonconstant irreducible polynomials $g_1,\ldots,g_r$ in $\mathbb{Z}[z]$, $1\leq r\leq m$. Then $1\leq |a_0|=|g(0)|=|g_1(0)|\cdots|g_r(0)|$, and so, $|g_i(0)|\geq 1$ for each $i=1,\ldots,r$. Assume on the contrary that $r>m-j\geq 1$. Then $r>1$ and there exists an index $t$ with $1\leq t\leq r$ such that all the zeros of $g_{t}$ lie inside the disk $|z|<1/|a_m|$. If $\alpha_{t}$ is the leading coefficient of $g_{t}$, then $|\alpha_t|\leq |a_m|$ . So, we may write  $g_{t}(z)=\pm \alpha_{t} \prod_{\theta}(z-\theta)$, where the product is over all zeros of $g_{t}$. We then have $1\leq |g_{t}(0)|=|\alpha_{t}|\prod_\theta |\theta|<|\alpha_{t}|\frac{1}{|a_m|^{\deg(g_{t})}}\leq 1$, which is a contradiction.
\end{proof}
\begin{proof}[\bf Proof of Theorem \ref{th:4}]
Define $g(z)=b^{m-1}f(z/b)$. Then we have
\begin{eqnarray*}
g&=&\sum_{0\leq i<j}a_ib^{m-1-i}z^i+(a_j b^{m-1-j}z^j)+ \sum_{j< i<m}a_ib^{m-1-i}z^i+(a_m/b)z^m.
\end{eqnarray*}
Since $b$ divides $|a_m|$, we must have $g\in \mathbb{Z}[z]$,  and $f$ and $g$ have same number of irreducible factors in $\mathbb{Z}[z]$. In view of this, it will be enough to prove the theorem for the polynomial $g$.
For any complex number $z$ satisfying  $|b/a_m|\leq |z|\leq 1$, we may have the following on using the hypothesis
\begin{eqnarray*}
|g(z)| &\geq & b^{m-1-j}|z|^j\Bigl(|a_j|-\sum_{0\leq i<j}|a_i|b^{j-i}\frac{1}{|z|^{j-i}}-\sum_{j<i\leq m}|a_i|b^{-(i-j)}|z|^{i-j}\Bigr)\\
 &\geq & b^{m-1-j}||b/a_m|^{j}\Bigl(|a_j|-\sum_{0\leq i<j}|a_i|b^{j-i}\frac{|a_m|^{j-i}}{b^{j-i}}-
 \sum_{j<i\leq m}|a_i|b^{-(i-j)}\Bigr)\\
 &=& \frac{b^{m-1}}{|a_m|^{j}}\Bigl(|a_j|-\sum_{0\leq i<j}|a_i||a_m|^{j-i}-\sum_{j<i\leq m}|a_i|b^{-(i-j)}\Bigr)\\
 &\geq & \frac{b^{m-1}}{|a_m|^{j}}\Bigl(|a_j|-\sum_{0\leq i<j}|a_i||a_m|^{j-i}-\sum_{j<i\leq m}|a_i|\delta^{i-j}\Bigr)>0,\\
 \end{eqnarray*}
which shows that each zero $\theta$ of $g$ satisfies $|\theta|<b/|a_m|\leq 1$ or $|\theta|>1$. Observe that for $|z|=1$,
we have
\begin{eqnarray*}
|a_j|b^{m-1-j}|z^j|=|a_j|b^{m-1-j}&>&b^{m-1-j}\Bigl(\sum_{0\leq i<j}|a_i||a_m|^{j-i}+\sum_{j<i\leq n}|a_i|b^{-(i-j)}\Bigr)\\
&\geq &b^{m-1-j}\Bigl(\sum_{0\leq i<j}|a_i|b^{j-i}+\sum_{j<i\leq n}|a_i|b^{-(i-j)}\Bigr)\\
&= &\sum_{0\leq i<j}|a_i|b^{m-1-i}+\sum_{j<i\leq n}|a_i|b^{m-1-i}\\
&\geq& |g(z)-a_jb^{m-1-j}z^j|_{|z|=1},
\end{eqnarray*}
where the polynomial $a_j b^{m-1-j}z^j$ has exactly $j$ zeros within the interior of the unit circle $|z|=1$. This in view of  Rouch\'e's theorem  tells us that there are exactly $j$ zeros of $g$ inside the disk $|z|<|b/a_m|\leq 1$. Thus the remaining  $m-j$ zeros of $g$ must lie outside of the closed disk $|z|\leq 1$. By Lemma \ref{L2}, the polynomial $g$ is a product of at most $m-j$ irreducible polynomials in $\mathbb{Z}[z]$.
\end{proof}


\subsection*{Disclosure statement}
The authors report there are no competing interests to declare.
\bibliographystyle{ams}

\end{document}